\documentclass[12pt]{amsproc}
\usepackage{enumerate,graphicx,amssymb,xy}
\xyoption{all}
\usepackage{hyperref}
\theoremstyle{plain}
\newtheorem{theorem}{Theorem}[section]
\newtheorem*{theorem*}{Theorem}

\newtheorem*{prop*}{Proposition}
\newtheorem{lemma}[theorem]{Lemma}
\newtheorem*{lemma*}{Lemma}
\newtheorem{cor}[theorem]{Corollary}
\newtheorem*{cor*}{Corollary}

\theoremstyle{definition}

\newtheorem*{defn*}{Definition}
\newtheorem{remark}[theorem]{Remark}
\newtheorem*{remark*}{Remark}
\newtheorem{example}[theorem]{Example}
\newcommand{\N}{\mathbf N}
\newcommand{\Z}{\mathbf Z}

\renewcommand{\setminus}{-}
\title[Degenerations and Orbits]{Degenerations and orbits in\\finite abelian groups}
\author{Kunal Dutta}
\author{Amritanshu Prasad}
\address{The Institute of Mathematical Sciences, Chennai.}
\begin{document}
\begin{abstract}
  A notion of degeneration of elements in groups is introduced.
  It is used to parametrize the orbits in a finite abelian group under its full automorphism group by a finite distributive lattice.
  A pictorial description of this lattice leads to an intuitive self-contained exposition of some of the basic facts concerning these orbits, including their enumeration.
  Given a partition $\lambda$, the lattice parametrizing orbits in a finite abelian $p$-group of type $\lambda$ is found to be independent of $p$.
  The order of the orbit corresponding to each parameter, which turns out to be a polynomial in $p$, is calculated.
  The description of orbits is extended to subquotients by certain characteristic subgroups.
  Each such characteristic subquotient is shown to have a unique maximal orbit.
\end{abstract}
\maketitle
\section{Background}
\label{sec:background}
Let $A$ be a finite abelian group and $G$ be its group of automorphisms.
This paper concerns the set $G\backslash A$ of $G$-orbits in $A$.

For each prime number $p$, let $A_p$ denote the $p$-primary part of $A$ (elements of $A$ annihilated by some power of $p$).
Then $A$ is the direct sum of the subgroups $A_p$, and $G$ is the product of the automorphism groups $G_p$ of $A_p$.
In fact,
\begin{equation*}
  G\backslash A = \prod_p G_p\backslash A_p.
\end{equation*}
Each finite abelian $p$-group in turn is isomorphic to 
\begin{equation*}
  A_{\lambda,p}=\Z/p^{\lambda_1}\Z \oplus \cdots \oplus \Z/p^{\lambda_l}\Z
\end{equation*}
for a unique non-increasing sequence $\lambda=(\lambda_1\geq \cdots \geq \lambda_l)$ of positive integers (in other words, a partition).
Write $G_{\lambda,p}$ for the automorphism group of $A_{\lambda,p}$.
Therefore, in order to study $G\backslash A$, it suffices to study $G_{\lambda,p}\backslash A_{\lambda,p}$ for each prime number $p$ and each partition $\lambda$.

A well-known formula for the cardinality of $G_{\lambda,p}\backslash A_{\lambda,p}$ goes back more than a hundred years to Miller \cite{1905}:
\begin{equation}
  \label{eq:1}
  |G_{\lambda,p}\backslash A_{\lambda,p}|=(\lambda_l+1)\prod_{i=1}^{l-1} (\lambda_i-\lambda_{i+1}+1)
\end{equation}
Another formula that appears in the literature (Schwachh\"offer and Stroppel \cite{MR1656579}) is given as follows: let $\tau_1<\tau_2<\cdots<\tau_t$ be the distinct natural numbers occurring in the partition $\lambda$.
Then
\begin{equation}
  \label{eq:2}
  |G_{\lambda,p}\backslash A_{\lambda,p}|=\sum_{k=0}^t\:\:\sum_{1\leq i_1<\cdots<i_k\leq t}\: \tau_{i_k} \prod_{j=1}^{k-1}(\tau_{i_j}-\tau_{i_{j+1}}-1).
\end{equation}
In this formula, the summand corresponding to $k=0$ is taken as $1$.
When $p\neq 2$, the map taking each element in $A_{p,\lambda}$ to the smallest characteristic subgroup containing it is a bijection between $G_{\lambda,p}\backslash A_{\lambda,p}$ and the set of characteristic subgroups of $A_{\lambda,p}$ (subgroups which are invariant under $G_{\lambda,p}$) (see Baer \cite[Sections~3 and 4]{Baer35} and Birkhoff \cite[Section~10]{GarrettBirkhoff01011935}).
Characteristic subgroups, partially ordered by inclusion, form a finite distributive lattice.
By the Fundamental Theorem for Finite Distributive Lattices \cite[Theorem~3.4.1]{MR1442260} there is a unique poset of which this lattice is the lattice of order ideals.
This poset was used by Kerby and Turner \cite{0905.1885v1} to prove that the lattice of characteristic subgroups in $A_{p,\lambda}$ distinguishes between the $\lambda$'s (except $\lambda=(4,2,1)$ and $(5,2)$) and is independent of $p$ when $p\neq 2$.
\section{Outline of this article}
\label{sec:outline-this-article}
Central to this article is the notion of degeneration:
\begin{defn*}
  [Degeneration]
  If $A$ and $B$ are groups, $a\in A$ and $b\in B$, we say that $a$ degenerates to $b$ (denoted $a\to b$) if there exists a homomorphism $\phi:A\to B$ such that $\phi(a)=b$.
\end{defn*}
For finite abelian groups, any homomorphism $\phi:A\to B$ maps $A_p$ into $B_p$.
 Therefore, it suffices to fix a prime $p$ and restrict attention to finite abelian $p$-groups.
Assume without any loss of generality that  $A=A_{p,\lambda}$ and that $B=A_{p,\mu}$ for partitions $\lambda=(\lambda_1\geq\cdots\geq\lambda_l)$ and $\mu=(\mu_1\geq\cdots\geq\mu_m)$.
Theorem~\ref{theorem:degeneracy} gives necessary and sufficient conditions for an element of $A$ to degenerate to an element of $B$ in terms of order ideals in a certain poset which we call the \emph{fundamental poset}, which is independent of $p$.

For an element $a\in A$, let $[a]$ denote its orbit under the automorphism group $G$ of $A$.
Clearly, degeneration descends to a relation on orbits: if $[a]=[a']$ and $[b]=[b']$ then $a\to b$ if and only if $a'\to b'$.
Write $[a]\geq [b]$ if $a\to b$.
In the special case where $B=A$, we show that \lq$\geq$\rq{} is a partial order on $G\backslash A$.
More precisely for $A=A_{p,\lambda}$, we show (Theorem~\ref{theorem:orbits}) that this partially ordered set is the lattice of ideals in an induced subposet $P_\lambda$ of the fundamental poset determined by $\lambda$, and hence a distributive lattice which is independent of $p$ ($P_\lambda$ determines $\lambda$, except for $(4,2,1)$ and $(5,2)$, which have isomorphic $P_\lambda$'s \cite{0905.1885v1}).

Viewing orbits as ideals in $P_\lambda$ allows us to interpret their enumeration given by (\ref{eq:1}) and (\ref{eq:2}), the first as counting ideals in terms of their boundaries, and the second as counting them in terms of their antichains of maximal elements (Section~\ref{sec:enumeration}).

In Section~\ref{sec:subq-char-subgr}, following \cite{1905,Baer35,GarrettBirkhoff01011935} the lattice of orbits of $A_{p,\lambda}$ is embedded in the lattice of characteristic subgroups (this embedding is an isomorphism if $p\neq 2$).
A simple formula for the order of the characteristic subgroup associated to an orbit is obtained.
This formula, along with the M\"obius inversion formula, is used to show that the cardinality of the $G_{p,\lambda}$-orbit in $A_{p,\lambda}$ corresponding to an ideal $I$ in $P_\lambda$ is a monic polynomial in $p$ with integer coefficients (Theorem ~\ref{theorem:polynomial}).
This polynomial is computed by two different methods (Theorems~\ref{theorem:orbits-order} and~\ref{theorem:visual-order}).

In Section~\ref{sec:orbits-subq-char}, the results of Sections~\ref{sec:orbits-ideals} and~\ref{sec:orders-orbits} are extended to $G_{p,\lambda}$-orbits in a subquotient by characteristic subgroups associated to orbits.
A consequence of this combinatorial description is the existence of a unique maximal orbit in each subquotient (Section~\ref{sec:maximal-orbits}).
For large $p$, a randomly chosen element of the subquotient is most likely in the maximal orbit (Theorem~\ref{theorem:maximal-size}).

We were motivated by attempts to understand the decomposition of the Weil representation associated to a finite abelian group $A$.
The sum of squares of the multiplicities in the Weil representation is the number of orbits in $A\times \hat A$ under automorphisms of a symplectic bicharacter ($\hat A$ denotes the Pontryagin dual of $A$) \cite{gdft}.
Techniques developed in this paper can be carried over to the symplectic setting, giving results about Weil representations \cite{Wfg}.

Our results and their proofs remain valid for orbits in a finitely generated torsion-module over a Dedekind domain under the group of module automorphisms.
Generalizing in another direction, all the results in this article, except for those in Section~\ref{sec:orders-orbits}, Corollary~\ref{cor:orbit-order} and Theorem~\ref{theorem:maximal-size}, hold for abelian groups of bounded order (see \cite[Chapter~18]{Kaplansky} and \cite{MR1656579}).
\section{The Fundamental Poset}
\label{sec:fundamental-poset}
\begin{lemma}
  \label{lemma:fp}
  Let $u$, $v$, $r$, $s$, $k$ and $l$ be non-negative integers such that $u$ and $v$ are not divisible by $p$, $r<k$ and $s<l$.
  Then $p^r u\in \Z/p^k\Z$ degenerates to $p^s v\in \Z/p^l\Z$ if and only if $r\leq s$ and $k-r\geq l-s$.
  If in addition, $p^s v\in \Z/p^l\Z$ degenerates to $p^r u\in \Z/p^k\Z$ then $k=l$ and $r=s$.
\end{lemma}
\begin{proof}
  If $\phi:\Z/p^k\Z\to \Z/p^l\Z$ takes $p^r u$ to $p^s v$ then $p^s v=\phi(p^r u)=p^r\phi(u)$, from which it follows that $s\geq r$.
  Also,
  \begin{equation*}
    0=\phi(p^k u)=p^{k-r}\phi(p^r u)=p^{k-r}p^s v=p^{k-r+s}v.
  \end{equation*}
  Therefore $k-r+s\geq l$, giving $k-r\geq l-s$.

  Conversely, if $r\leq s$ and $k-r\geq l-s$ then there is a unique homomorphism $\Z/p^k\Z\to \Z/p^l\Z$ for which $1\mapsto v u^{-1}p^{s-r}$ (where $u^{-1}$ is a multiplicative inverse of $u$ modulo $p^l$).
  This homomorphism maps $p^r u$ to $p^s v$.

  The second part of the lemma is an immediate consequence of the first.
\end{proof}
The group $\Z/p^k\Z$ has $k$ orbits of non-zero elements under the action of its automorphism group, represented by $1,p,\ldots,p^{k-1}$.
Let $P$ be the disjoint union over all $k\in \N$ of the orbits of non-zero elements in $\Z/p^k\Z$.
We denote the orbit of $a$ in $\Z/p^k\Z$ by $(a,k)$.
Degeneracy descends to a relation on $P$.
It follows from Lemma~\ref{lemma:fp} that this reflexive and transitive relation is in fact a partial order.
The Hasse diagram for $P$, which we call the fundamental poset, is given in Figure~\ref{fig:funda}.
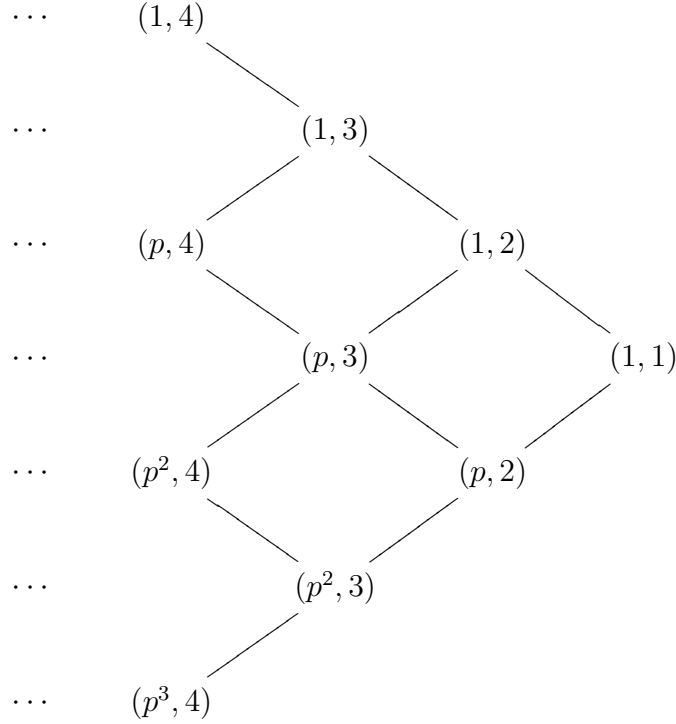
\begin{figure}
  \centering
  \begin{equation*}
    \xymatrix{
      \cdots & (1,4) \ar@{-}[dr] & & & \\
      \cdots && (1,3) \ar@{-}[dr] & & \\
      \cdots &(p,4) \ar@{-}[dr] \ar@{-}[ur] & & (1,2) \ar@{-}[dr] & \\
      \cdots && (p,3) \ar@{-}[dr] \ar@{-}[ur] & & (1,1)\\
      \cdots &(p^2,4) \ar@{-}[dr] \ar@{-}[ur] & & (p,2) \ar@{-}[ur] & \\
      \cdots && (p^2,3) \ar@{-}[ur] & & \\
      \cdots &(p^3,4) \ar@{-}[ur] & & &\\
    }
  \end{equation*}
  \caption{The fundamental poset $P$}
  \label{fig:funda}
\end{figure}
\begin{remark}
Let $P_n$ denote the induced subposet of $P$ consisting of the orbits in the groups $\Z/p^k\Z$ for $k=1,\ldots,n$.
One of the characterizations of the Catalan numbers is as the number of lattice paths from the maximal element of $P_n$ to its minimal element \cite[Example~3.5.5]{MR1442260}.
Say that a degeneration $a\to b$ is strict if $b$ does not degenerate to $a$.
Thus the number of maximal chains of strict degenerations of orbits in the groups $\Z/p^k\Z$, $k=1,\dots,n$ is a Catalan number.
\end{remark}
\section{Degeneracy and ideals}
Given $a=(a_1,\ldots,a_l)\in A_{p,\lambda}$, we refer to as the \emph{ideal of $a$}, the ideal in $P$ generated by the orbits of the non-zero coordinates $a_i\in \Z/p^{\lambda_i}\Z$ of $a$. 
We denote this ideal by $I(a)$ (see Example~\ref{example:ideals}).
\begin{theorem}
  \label{theorem:degeneracy}
  Given partitions $\lambda$ and $\mu$, $a\in A_{p,\lambda}$ and $b\in A_{p,\mu}$, $a$ degenerates to $b$ if and only if $I(a)\supset I(b)$.
\end{theorem}
\begin{proof}
  Take $\lambda=(\lambda_1\geq\cdots\geq\lambda_l)$, $\mu=(\mu_1\geq\cdots\geq\mu_m)$.   
  If $I(a)\supset I(b)$ then for each coordinate $b_i\in \Z/p^{\mu_i}\Z$ of $b$ there exists a coordinate $a_j\in \Z/p^{\lambda_j}\Z$ and a homomorphism $\phi_{ij}:\Z/p^{\lambda_j}\Z\to \Z/p^{\mu_i}\Z$ such that $\phi_{ij}(a_j)=b_i$.
  For each $i$, choose such a $j$ and $\phi_{ij}$.
  For all other pairs $(i,j)$ let $\phi_{ij}=0$.
  Then the homomorphism $A_{p,\lambda}\to A_{p,\mu}$ with matrix given by $(\phi_{ij})$ then takes $a$ to $b$.

Conversely, suppose that $\phi:A_{p,\lambda}\to A_{p,\mu}$ is a homomorphism such that $\phi(a)=b$.
Then
\begin{equation*}
  b_i=\sum_{j=1}^l\phi_{ij}(a_j)
\end{equation*}
for homomorphisms $\phi_{ij}:\Z/p^{\lambda_j}\Z\to \Z/p^{\mu_i}\Z$.
It is clear that, in a cyclic abelian group, any sum is a degeneration of at least one of its terms.
Therefore, using transitivity of degeneration, if $b_i\neq 0$, then it is a degeneration of at least one entry of $a$, from which it follows that $I(b)\subset I(a)$.
\end{proof}
\begin{example}
  \label{example:ideals}
  Let $\lambda=\mu=(7,5,3,3,2)$, $a=(p^5,p,p^2,1,p)$ and $b=(p^4,p^4,p,p,0)$.
  The ideals $I(a)$ are $I(b)$ of $P$ are shown in Figure~\ref{fig:ideals}.
  It follows from Theorem~\ref{theorem:degeneracy} that $a$ degenerates to $b$.
\end{example}
\begin{figure}
  \centering
  \begin{tabular}{cc}
    \includegraphics[width=0.4\textwidth, trim = 0.5cm 23.05cm 17.4cm 1.05, clip]{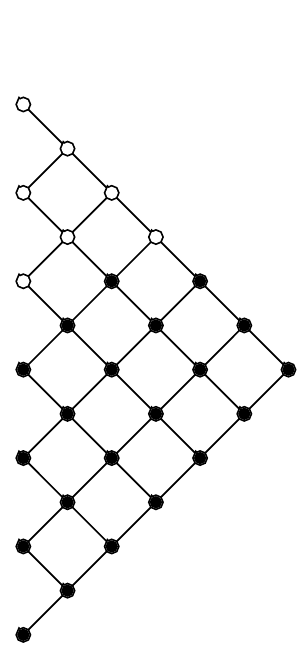}&
    \includegraphics[width=0.4\textwidth, trim = 0.5cm 20cm 17.45cm 4.05cm, clip]{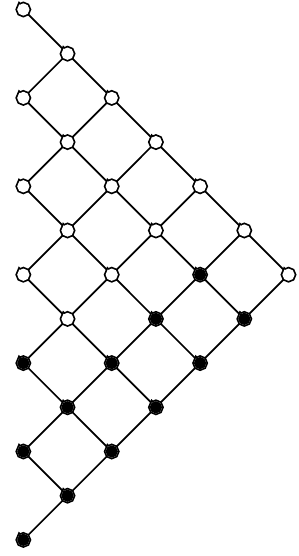}\\
    $I(a)$ & $I(b)$
  \end{tabular}
  \caption{Ideals associated to elements}
  \label{fig:ideals}
\end{figure}
\section{Orbits and ideals}
\label{sec:orbits-ideals}
Given a partition $\lambda=(\lambda_1\geq\cdots\geq\lambda_l)$, let $P_{\lambda}$ be the induced subposet of $P$ consisting of orbits of the form $(a,k)$ (notation as in Section~\ref{sec:fundamental-poset}) with $k=\lambda_i$ for some $i=1,\ldots,l$.
\begin{example}
  For $\lambda=(7,5,3,3,2)$ the construction of the poset $P_\lambda$ is shown in Figure~\ref{fig:construction}.
\end{example}
\begin{figure}
  \centering
  \begin{tabular}{cc}
    \includegraphics[width=0.4\textwidth, trim= 0.5cm 19.7cm 17cm 3cm, clip]{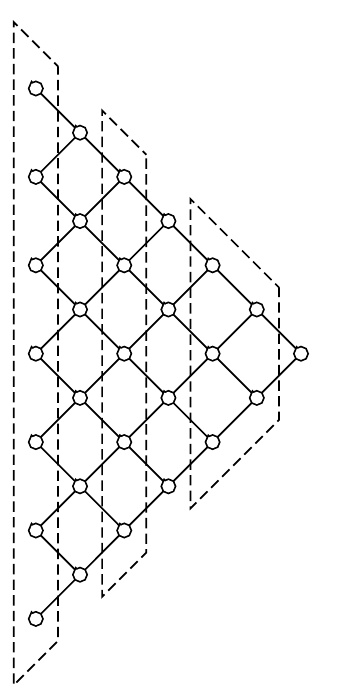}&
    \includegraphics[width=0.4\textwidth, trim = 0.5cm 19.7cm 17cm 3cm, clip]{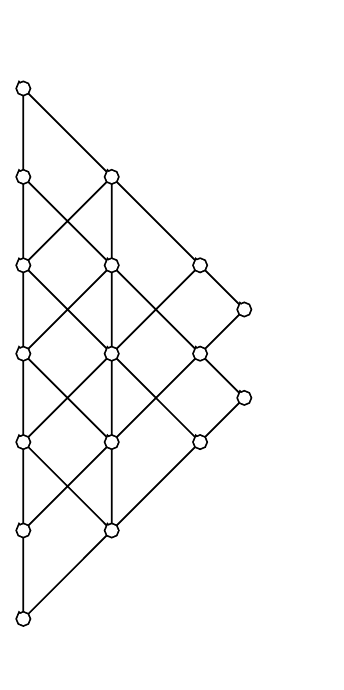}\\
    Column selection & The poset $P_{7,5,3,3,2}$
  \end{tabular}
  \caption{Construction of $P_\lambda$}
  \label{fig:construction}
\end{figure}
Given an ideal $I$ in $P$, $I\cap P_{\lambda}$ is an ideal in $P_{\lambda}$.
The following lemma is true for the intersection of ideals with an induced subposet of any poset:
\begin{lemma}
  \label{lemma:intersections}
  The map taking an ideal in $P$ to its intersection with $P_\lambda$ is a bijection between ideals in $P$ which are generated by elements of $P_\lambda$ and ideals in $P_\lambda$.
\end{lemma}
\begin{proof}
  Recall that there is a one to one correspondence between antichains and ideals; the maximal elements in an ideal form an antichain that generates the ideal (see \cite[Section~3.1]{MR1442260}).
  Map an ideal in $P_\lambda$ to the ideal in $P$ generated by its maximal elements.
  This is an inverse to $I\mapsto I\cap P_\lambda$.
\end{proof}
For an element $a\in A_{p,\lambda}$ let $[a]$ denote its $G_{p,\lambda}$-orbit.
For each $a\in A_{p,\lambda}$ let $I_\lambda(a)=I(a)\cap P_\lambda$.
If $[a]=[b]$, then by Theorem~\ref{theorem:degeneracy}, $I_\lambda(a)=I_\lambda(b)$.
Conversely,
\begin{lemma}
  \label{lemma:antisym}
  If $I_\lambda(a)=I_\lambda(b)$ then $[a]=[b]$.
\end{lemma}
\begin{proof}
  Define a function $e:P_\lambda\to A_{p,\lambda}$ as follows: take $x=(p^r,k)$ to the unique vector $e(x)\in A_{p,\lambda}$ all of whose entries $e(x)_i$ are $0$ except for the least $i$ such that $\lambda_i=k$, when $e(x)_i=p^r$.
  Let $I\subset P_\lambda$ be an ideal.
  Let $\max(I)$ denote the antichain of maximal elements of $I$.
  Let
  \begin{equation*}
    a(I)=\sum_{x\in \max(I)} e(x).
  \end{equation*}
  Now take $I=I_\lambda(a)$.
  That $[a]=[a(I)]$ is shown by reducing $a$ to $a(I)$ by a sequence of automorphisms.
  Firstly, after applying an automorphism to each coordinate of $a$, one may assume that it is a power of $p$.
  For each $i$ such that $a(I)_i\neq 0$, there exists $j$ such that $\lambda_i=\lambda_j$ and $a_j=a(I)_i$.
  Since $\lambda_i=\lambda_j$, interchanging the $i$th and $j$th coordinates is an automorphism of $A_{p,\lambda}$.
  Therefore, one may assume that for those $i$ such that $a(I)_i\neq 0$, $a(I)_i=a_i$.
  If $a(I)_i=0$, $y=(a_i,\lambda_i)\in P_\lambda$ lies in $I$, so that there exists a maximal element $x$ of $I$ such that $x\geq y$. 
  In other words, there exists $j\neq i$ and a homomorphism $\phi_{ij}:\Z/p^{\lambda_j}\Z\to \Z/p^{\lambda_i}\Z$ such that $\phi_{ij}(a(I)_j)=a_i$.
  The automorphism of $A_{p,\lambda}$ given by
  \begin{equation*}
    (b_1,\ldots,b_l)\mapsto (b_1,\ldots,b_i-\phi_{ij}(b_j),\ldots,b_l)
  \end{equation*}
  kills the $i$th entry of $a$ while leaving all the others unchanged.
  One may repeat this process until $a$ is reduced to $a(I)$, so that $[a]=[a(I)]$.

  Now if $I_\lambda(b)=I$ as well, we have $[a]=[a(I)]=[b]$.
\end{proof}
For $a,b\in A_{p,\lambda}$ say that $[a]\geq [b]$ if $a\to b$.
This is clearly a well-defined reflexive and transitive relation on $G_{p,\lambda}\backslash A_{p,\lambda}$.
Lemma~\ref{lemma:antisym}, together with Theorem~\ref{theorem:degeneracy} implies that it is also antisymmetric, and that if $J(P_\lambda)$ denotes the lattice of ideals in $P_\lambda$ then
\begin{theorem}
  \label{theorem:orbits}
  The relation \lq$\geq$\rq{} is a partial order on $G_{p,\lambda}\backslash A_{p,\lambda}$.
  The map $\phi:G_{p,\lambda}\backslash A_{p,\lambda}\to J(P_\lambda)$ given by $\phi([a])=I_{\lambda}(a)$ is an isomorphism of posets.
\end{theorem}
\section{Enumeration}
\label{sec:enumeration}
Let $I\subset P_\lambda$ be an ideal.
For each $i$ let $r_i$ be the smallest non-negative integer such that $(p^{r_i},\lambda_i)\in I$.
Clearly the ordered set $\mathbf r=\langle r_1,r_2,\ldots, r_l\rangle$ determines $I$. 
Therefore to count the number of ideals, we only need to count the number of ordered sets $\mathbf r$ such that
\begin{equation*}
  I(\mathbf r)=\{(p^{s_i},\lambda_i)|1\leq i\leq l,\;s_i\geq r_i\}
\end{equation*}
is an ideal in $P_\lambda$.
For any $i<l$, if $r_{i-1}>r_i+(\lambda_{i-1}-\lambda_i)$, then $I(\mathbf r)$ would not be an ideal, since it would contain $x=(p^{r_i},\lambda_i)$ but not $y=(p^{r_i+(\lambda_{i-1}-\lambda_i)},\lambda_{i-1})$, even though $x\geq y$.
Similarly, if $r_{i-1}<r_i$, $I(\mathbf r)$ would not be an ideal.
Therefore, if $I(\mathbf r)$ is an ideal, then
\begin{equation}
  \label{eq:3}
  r_i\leq r_{i-1}\leq r_i+(\lambda_{i-1}-\lambda_i) \text{ for } i=1,\ldots,l-1.
\end{equation}
Furthermore, when all the above inequalities are satisfied, then $I(\mathbf r)$ is an ideal.
Therefore one way of counting the number of ideals is as follows: $r_l$ (and hence the element $(p^{r_l},\lambda_l)$) can be chosen in $\lambda_l+1$ ways; once $r_{i+1},\ldots,r_l$ have been chosen, there are $(\lambda_{i-1}-\lambda_{i}+1)$ choices for $r_{i-1}$.
The total number of sequences $r_1,\ldots r_l$ satisfying (\ref{eq:3}) is therefore given by (\ref{eq:1}).

The piecewise linear curve in the Hasse diagram of $P_\lambda$ obtained by joining $(p^{r_i},\lambda_i)$ as $i$ increases from $1$ to the largest $i$ for which $r_i<\lambda_i$ can be regarded as the the boundary of the ideal $I$, and denoted by $\partial I$.
\begin{example}
  For $\lambda$, $a$ and $b$ as in Example~\ref{example:ideals}, the boundaries of $I_\lambda(a)$ and $I_\lambda(b)$ are the piecewise linear curves represented by solid lines in Figure~\ref{fig:boundaries}.
\end{example}
\begin{figure}
  \centering
  \begin{tabular}{cc}
    \includegraphics[width=0.3\textwidth, trim = 3.3cm 4.7cm 15cm 19cm, clip]{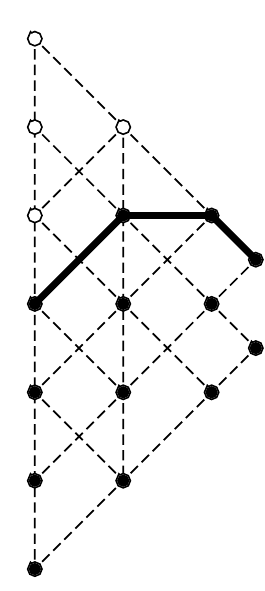}\hspace{0.5cm}&\hspace{0.5cm}
    \includegraphics[width=0.35\textwidth, trim = 0.58cm 19.9cm 17.1cm 3cm, clip]{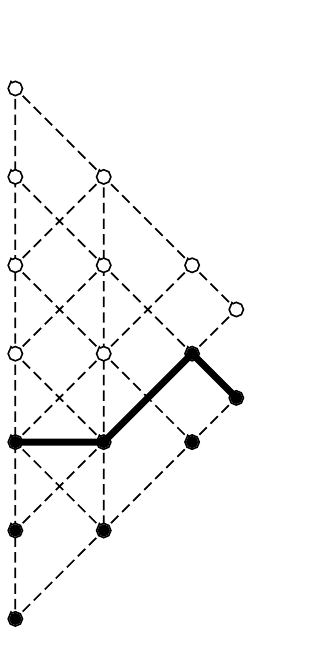}\\
    $\partial I_\lambda(a)$ & $\partial I_\lambda(b)$\\
  \end{tabular}
  \caption{Boundaries of ideals}
  \label{fig:boundaries}
\end{figure}
Another way to count ideals is by using the correspondence between antichains and ideals mentioned in the proof of Lemma~\ref{lemma:intersections}.
Any antichain in $P_\lambda$ can have at most one entry in each column.
As in Section~\ref{sec:background}, let $\tau_1<\cdots<\tau_t$ be the distinct natural numbers occurring in the partition $\lambda$.
To count the number of antichains with one entry in each of the columns representing orbits in $\Z/p^{\tau_{i_j}}\Z$ for some sequence $1\leq i_1<\cdots<i_k\leq t$, start with the column of orbits in $\Z/p^{\tau_{i_k}}\Z$, where there are $\tau_{i_k}$ choices.
The number of elements in the column corresponding to $\Z/p^{\tau_{i_{k-1}}}\Z$ which are not comparable with the previously chosen entry is precisely $\tau_{i_{k-1}}-\tau_{i_k}-1$.
Continuing thus, and then summing over all possible subsets $1\leq i_1<\cdots<i_k\leq m$ gives (\ref{eq:2}).
\begin{example}
  \label{eg:maximals}
  For $\lambda$, $a$ and $b$ as in Example~\ref{example:ideals}, the maximal elements of $I_\lambda(a)$ are $(p,5)$ and $(1,3)$, whereas for $I_\lambda(b)$ they are $(p^4,7)$ and $(p,3)$. The total number of orbits in $A_{p,\lambda}$ is readily seen to be $54$ by either method.
\end{example}
\section{Characteristic subgroups}
\label{sec:subq-char-subgr}
For each ideal $I\subset P_\lambda$ consider the subset
\begin{equation*}
  A_{p,I}=\{a\in A_{p,\lambda}: I_\lambda(a)\subset I\}
\end{equation*}
Clearly, $A_{p,I}$ is a characteristic subgroup of $A_{p,\lambda}$.
Moreover, $J\subset I$ if and only if $A_{p,J}\subset A_{p,I}$.
It follows that 
\begin{theorem}
  \label{theorem:orbits-to-subgroups}
  The function $[a]\mapsto A_{p,I_\lambda(a)}$ embeds the lattice of orbits in $A_{p,\lambda}$ into the lattice of characteristic subgroups of $A_{p,\lambda}$.
\end{theorem}
\begin{remark}
  When $p\neq 2$, these sets have the same cardinality \cite[Section~10]{GarrettBirkhoff01011935}, so that $[a]\mapsto A_{p,I_\lambda(a)}$ is an isomorphism of posets.
\end{remark}
Given $x\in P_\lambda$, let $m(x)$ denote its multiplicity: if $x=(a,k)$ then $m(x)$ is the number of times that $k$ occurs in the partition $\lambda$.
For an ideal $I\subset P_\lambda$ let $[I]$ denote the number of points in it, counted with multiplicity:
\begin{equation*}
  [I]=\textstyle{\sum_{x\in I}} m(x)
\end{equation*}
For example, when $\lambda=(7,5,3,3,2)$, the elements in $P_\lambda$ of the form $(p^r,3)$ are counted twice. All other elements are counted once.
\begin{theorem}
  \label{theorem:order}
  For each ideal $I\subset P_\lambda$,
  $|A_{p,I}|=p^{[I]}$.
\end{theorem}
\begin{proof}
  Note that
  \begin{equation*}
    A_{p,I}=\{(a_1,\ldots,a_l)\in A_{p,\lambda}| [a_i]\in I \text{ for each } i\},
  \end{equation*}
  from which the result follows.
\end{proof}
\begin{example}
  \label{eg:char-subgroups}
  For $\lambda$, $a$ and $b$ as in Example~\ref{example:ideals}, that $|A_{p,I_\lambda(a)}|=p^{16}$ and $|A_{p,I_\lambda(b)}|=p^{10}$ is evident from Theorem~\ref{theorem:order} and Figure~\ref{fig:boundaries}.
\end{example}
\section{Orders of orbits}
For each ideal $I$ in $P_\lambda$, let $O_{p,I}$ denote the orbit in $A_{p,\lambda}$ corresponding to $I$ under the bijection of Theorem~\ref{theorem:orbits}.
Thus
\begin{equation}
  \label{eq:4}
  O_{p,I}=\{a\in A_{p,\lambda}|I_\lambda(a)=I\}.
\end{equation}
\label{sec:orders-orbits}
\begin{theorem}
  \label{theorem:polynomial}
  For each ideal $I$ of $P_\lambda$, the order of $O_{p,I}$ is a monic polynomial in $p$ of degree $[I]$ with integer coefficients.
\end{theorem}
\begin{proof}
  Since
  \begin{equation*}
    p^{[I]}=|A_{p,I}|=\sum_{J\subset I}|O_{p,J}|,
  \end{equation*}
  the M\"obius inversion formula \cite[Section~3.7]{MR1442260} gives
  \begin{equation*}
    |O_{p,I}|=\sum_{J\subset I}\mu(J,I)p^{[I]}
  \end{equation*}
  where $\mu$ denotes the M\"obius function of the lattice of ideals in $P_\lambda$. Since $\mu$ is integer-valued with $\mu(I,I)=1$ for each $I$, the result follows.
\end{proof}
Recall that the M\"obius function of any finite distributive lattice can be computed explicitly \cite[Example~3.9.6]{MR1442260}:
\begin{lemma}
  \label{lemma:Moebius}
  For ideals $J\subset I$ in $P_\lambda$,
  \begin{equation*}
    \mu(J,I)=
    \begin{cases}
      (-1)^{|I\setminus J|} & \text{ if } I\setminus J\text{ is an antichain}\\
      0 & \text{ otherwise.}
    \end{cases}
  \end{equation*}
\end{lemma}
Note that $I\setminus J$ is an antichain if and only if $I\setminus J$ is contained in the set $\max I$ of maximal elements of $I$.
We get an expression for the cardinality of an orbit:
\begin{theorem}
  \label{theorem:orbits-order}
  For every ideal $I$ of $P_\lambda$,
  \begin{equation*}
    |O_{p,I}|=\sum_{I\setminus \max I\subset J\subset I}(-1)^{|I\setminus J|}p^{[J]}.
  \end{equation*}
\end{theorem}
\begin{example}
  \label{eg:orbits-order}
  Let $\lambda$, $a$ and $b$ be as in Example~\ref{example:ideals}.
  Then 
  \begin{gather*}
    |O_{p,I(a)}|=p^{16}-p^{15}-p^{14}+p^{13}\\
    |O_{p,I(b)}|=p^{10}-p^9-p^8+p^7.
  \end{gather*}
\end{example}
The expression (\ref{eq:4}) for $O_{p,I}$ suggests a visual method for computing $|O_{p,I}|$.
Since $I_\lambda(a)$ is the ideal generated by the orbits of the coordinates of $a$ in their respective cyclic groups, $I_\lambda(a)=I$ if and only if each coordinate of $a$ has orbit in $I$ and for each maximal element $x$ of $I$ there is at least one coordinate of $a$ whose orbit is $x$.
Therefore
\begin{theorem}
  \label{theorem:visual-order}
  For every ideal $I\subset P_\lambda$,
  \begin{equation*}
    |O_{p,I}|=p^{[I]}\prod_{x\in \max I}\big(1-p^{-m(x)}\big)
  \end{equation*}
\end{theorem}
\begin{example}
  \label{eg:visual-order}
  Let $\lambda$, $a$ and $b$ be as in Example~\ref{example:ideals}.
  Using the data from Examples~\ref{eg:maximals} and~\ref{eg:char-subgroups},
  \begin{gather*}
    |O_{p,I(a)}|=p^{16}(1-p^{-1})(1-p^{-2})\\
    |O_{p,I(b)}|=p^{10}(1-p^{-1})(1-p^{-2})
  \end{gather*}
  consistent with Example~\ref{eg:orbits-order}.
\end{example}
\section{Orbits on subquotients of characteristic subgroups}
\label{sec:orbits-subq-char}
If $S\subset S'\subset A_{p,\lambda}$ are characteristic subgroups, then the action of $G_{p,\lambda}$ on $A_{p,\lambda}$ descends to an action on the subquotient $S'/S$ of $A_{p,\lambda}$.
\begin{theorem}
  \label{theorem:subquotient}
  Let $J\subset I\subset P_\lambda$ be ideals.
  Then there is a canonical bijective correspondence between the ideals in $P_\lambda$ which are contained in $I$ and contain $J$ (equivalently, ideals in the induced subposet $I\setminus J$) and the $G_{p,\lambda}$-orbits in $A_{p,I}/A_{p,J}$.
\end{theorem}
\begin{proof}
  The map $a\mapsto a+A_{p,J}$ induces a surjective function $\phi$ from the set of $G_{p,\lambda}$-orbits in $A_{p,I}$ to the set of $G_{p,\lambda}$-orbits in $A_{p,I}/A_{p,J}$.

  Suppose $I'\subset I$ is an ideal in $P_\lambda$.
  Since $\max(I'\cup J)\subset \max I'\cup \max J$, $\max(I'\cup J)\setminus \max I'\subset J$.
  It follows that (in the notation of the proof of Lemma~\ref{lemma:antisym}) $a(I'\cup J)-a(I')\subset A_{p,J}$.
  Therefore, $\phi(O_{p,I'})=\phi(O_{p,I'\cup J})$.
  Thus $\phi$ is surjective when restricted to orbits corresponding to ideals containing $J$.

  Now suppose $I_1$ and $I_2$ are distinct ideals of $P_\lambda$ which are contained in $I$ and contain $J$.
  If $I_1$ is not contained in $I_2$, then $O_{p,I_1}$ is not contained in $A_{p,I_2}$.
  Therefore, $\phi(O_{p,I_1})$ can not equal $\phi(O_{p,I_2})$, which is contained in $A_{p,I_2}/A_{p,J}$.
  Similarly, if $I_2$ is not contained in $I_1$, then $\phi(O_{p,I_1})\neq \phi(O_{p,I_2})$.

  Thus $\phi$, when restricted to orbits corresponding to ideals containing $J$, gives rise to the required bijection.
\end{proof}
\begin{example}
  For $\lambda$, $a$ and $b$ as in Example~\ref{example:ideals}, the number of $G_{p,\lambda}$-orbits in $A_{p,I(a)}/A_{p,I(b)}$ is $13$, since the poset $I(a)\setminus I(b)$ has Hasse diagram as in Figure~\ref{fig:subquot}.
\end{example}
\begin{figure}
  \centering
  \begin{tabular}{cc}
    \includegraphics[width=0.4\textwidth, trim = 0.5cm 19.7cm 17cm 3cm, clip]{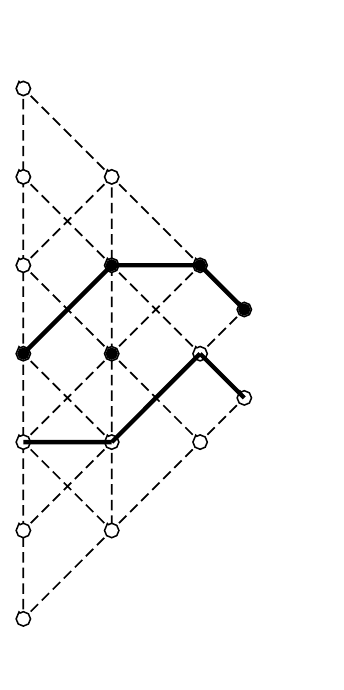}
  &\hspace{1cm}
  \includegraphics[width=0.4\textwidth, trim = 0.5cm 19.7cm 17cm 3cm, clip]{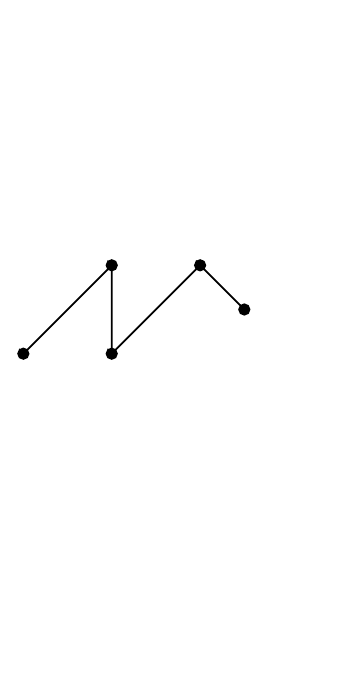}\\
  Inside $P_{7,5,3,3,2}$&\hspace{.7cm} By itself
\end{tabular}
  \caption{Hasse diagram of $I(a)\setminus I(b)$}
  \label{fig:subquot}
\end{figure}
\begin{cor}
  \label{cor:orbit-order}
  Let $J\subset I'\subset I$ be ideals in $P_\lambda$.
  Then the order of the $G_{p,\lambda}$-orbit in $A_{p,I}/A_{p,J}$ corresponding to $I'$ under the correspondence of Theorem~\ref{theorem:subquotient} is
  \begin{equation*}
    \frac{\textstyle{\sum_{I^{\prime\prime}\cup J=I'}}|O_{p,I^{\prime\prime}}|}{|A_{p,J}|}.
  \end{equation*}
\end{cor}
\begin{proof}
  By the proof of Theorem~\ref{theorem:subquotient}, the pre-image in $A_{p,I}$ of this orbit is the union of the orbits $O_{p,I^{\prime\prime}}$ as $I^{\prime\prime}$ ranges over ideals such that $I^{\prime\prime}\cup J=I'$, whence the result follows.
\end{proof}
\section{Maximal orbits}
\label{sec:maximal-orbits}
Each subquotient of the form $A_{p,I}/A_{p,J}$, where $J\subset I\subset P_\lambda$ are ideals, has a unique maximal $G_{p,\lambda}$-orbit which is not contained in $A_{p,I'}/A_{p,J}$ for any $J\subset I'\subsetneq I$, namely $\phi(O_{p,I})$ ($\phi$ is defined in the proof of Theorem~\ref{theorem:subquotient}).
In particular, $A_{p,\lambda}$ has a unique orbit (corresponding to the ideal $I=P_\lambda$) which does not intersect any proper subgroup of the form $A_{p,I}$.

A maximal orbit is, in a sense, dense:
\begin{theorem}
  \label{theorem:maximal-size}
  For each partition $\lambda$
  \begin{equation*}
    \lim_{p\to \infty} \frac{|\phi(O_{p,I})|}{|A_{p,I}/A_{p,J}|}=1
  \end{equation*}
\end{theorem}
\begin{proof}
  This is a consequence of Theorem~\ref{theorem:visual-order} and the observation that $|\phi(O_{p,I})|\geq|O_{p,I}|/|A_{p,J}|$ (which follows from Corollary~\ref{cor:orbit-order}).
\end{proof}
In particular, if $I'\subsetneq I$ are ideals in $P_\lambda$, then $|O_{p,I'}|/|O_{p,I}|\to 0$ as $p\to \infty$.
\subsection*{Acknowledgements}
We thank C. P. Anil Kumar and Pooja Singla for some helpful remarks.
\bibliographystyle{abbrv}
\bibliography{refs}
\end{document}